\documentclass{amsart}
\usepackage{amsfonts}
\usepackage{amssymb}
\usepackage{amsmath}

\newtheorem{theorem}{Theorem}[section]

\newtheorem{proposition}{Proposition}[section]
\newtheorem{corollary}{Corollary}[section]

\newtheorem{definition}{Definition}[section]
\newtheorem{example}{Example}[section]
\numberwithin{equation}{section}
\def\({\left ( }
\def\){\right )}
\def\<{\left < }
\def\>{\right >}

\begin{document}
\title[half-lightlike submanifolds with planar normal section in $R_{2}^{4}$]%
{half-lightlike submanifolds with planar normal sections in $R_{2}^{4}$}
\author{Feyza Esra Erdo\u{g}an}
\address{Faculty of Arts and Science, Department of Mathematics, Ad\i yaman
University, 02040 Ad\i yaman, TURKEY}
\email{ferdogan@adiyaman.edu.tr}
\author{R\i fat G\"{u}ne\c{s}}
\address{Faculty of Arts and Science, Department of Mathematics, \.{I}n\"{o}n%
\"{u} University, 44280\ \ Malatya, TURKEY}
\email{rifat.gunes@inonu.edu.tr}
\author{Bayram \c{S}ah\.{I}n}
\address{Faculty of Arts and Science, Department of Mathematics, \.{I}n\"{o}n%
\"{u} University, 44280\ \ Malatya, TURKEY}
\email{bayram.sahin@adiyaman.edu.tr}

\begin{abstract}
In the present paper\textbf{\ }we give condition for a half-lightlike
hypersurfaces of $R_{2}^{4}$ to have degenerate or non-degenerate planar
normal sections.
\end{abstract}

\maketitle

\section{Introduction}

Surfaces with planar normal sections\ in Euclidean spaces were first studied
by Bang-Yen Chen \cite{Bang-yen Chen}. In \cite{Young Ho Kim2}, Y. H. Kim
initiated the study of semi-Riemannian setting of such surfaces. Both
authors obtained similar results in these spaces. But as far as we know,
half-lightlike submanifold with planar normal sections have not been studied
so far. In this paper we study half-lightlike hypersurfaces with planar
normal sections in $R_{2}^{4}.$

Let $M$\ be a hypersurface in $R_{2}^{4}$. For a point $p$\ in $M$\ and a
lightlike vector $\xi $\ tangent to $M$\ at $p$\ which span radical
distribution, the vector $\xi $\ and transversal space $tr(TM)$\ to $M$\ at $%
p$\ determine a 2- dimensional subspace $E(p,\xi )$\ in $R_{2}^{4}$\ through
$p$. The intersection of $M$\ and $E(p,\xi )$\ gives a lightlike curve $%
\gamma $\ in a neighborhood of $p,$\ which is called the normal section of $%
M $\ at the point $p$\ in the direction of $\xi $. Let $w$\ be a spacelike
vector tangent to $M$\ at $p$\ $\left( v\in S(TM\right) )$. Then the vector $%
w$\ and transversal space $tr(TM)$\ to $M$\ at $p$\ determine a 2-
dimensional subspace $E(p,v)$\ in $R_{2}^{4}$\ through $p$. In this case,
the intersection of $M$\ and $E(p,v)$\ gives a spacelike curve $\gamma $\ in
a neighborhood of $p$\ which is called the normal section of $M$\ at $p$\ in
the direction of $v.$\ According to both situation given above, $M$\ is said
to have degenerate pointwise and non-degenerate pointwise planar normal
sections, respectively if each normal section $\gamma $\ at $p$\ satisfies $%
\gamma ^{\prime }\wedge \gamma ^{\prime \prime }\wedge \gamma ^{\prime
\prime \prime }=0$\ \cite{Bang-yen Chen}, \cite{Shi-jie Li}, \cite{Young Ho
Kim1}, \cite{Young Ho Kim3}.

\section{Preliminaries}

Let $\left( \bar{M},\bar{g}\right) $\ be an $\left( m+2\right) $-dimensional
semi-Riemannian manifold of index $q\geq 1$\ and $(M,g)$\ a lightlike
submanifold of codimension 2 of $\bar{M}.$\ Since $g$\ is degenerate, there
exists locally a vector field $\xi \in \Gamma \left( TM\right) $. Then, for
each tangent space $T_{x}M$\ we consider%
\begin{equation*}
T_{x}M^{\perp }=\left\{ u\in T_{x}\bar{M}:\bar{g}\left( u,v\right)
=0,\forall v\in T_{x}M\right\}
\end{equation*}%
which is a degenerate 2-codimensional subspace of $T_{x}\bar{M}$. Since $M$\
is lightlike, both $T_{x}M$\ and $T_{x}M^{\perp }$\ are degenerate
orthogonal subspaces but no longer complementary. In this case the dimension
of $RadT_{x}M=T_{x}M\cap T_{x}M^{\perp }$\ depends on the point $x\in M$. We
denote the radical distribution of a lightlike submanifold by $RadTM$. Then
there exists a complementary non-degenerate distribution $S(TM)$\ to $RadTM$%
\ in $TM$, called a screen distribution of $M$, with the orthogonal
distribution%
\begin{equation*}
TM=RadTM\oplus _{orth}S(TM).
\end{equation*}

\begin{definition}
The submanifold $(M,g,S(TM))$\ is called a half-lightlike submanifold if $%
\dim (RadTM)=1.$\ The term half-lightlike has been used since for this class
$\left( TM\right) ^{\perp }$\ is half lightlike. On the other hand, if $\dim
(RadTM)=2$, then, $RadTM=\left( TM\right) ^{\perp }$\ and $(M,g,S(TM))$\ is
called a co-isotropic submanifold. In this section, we present results on
half-lightlike submanifolds for which there exist $\xi ,u\in T_{x}M^{\perp }$%
\ such that\
\begin{equation*}
\bar{g}\left( \xi ,v\right) =0,\text{ \ \ }\bar{g}\left( u,u\right) \neq 0,%
\text{ \ \ }\forall v\in T_{x}M^{\perp }.
\end{equation*}%
The above relations imply that $\xi \in $\ $T_{x}M$, so $\xi \in RadT_{x}M.$%
\ Therefore, locally there exists a lightlike vector field $\xi $\ on $M$\
such that%
\begin{equation*}
\bar{g}\left( \xi ,X\right) =\bar{g}\left( \xi ,u\right) =0,\text{ \ \ }%
\forall X\in \Gamma \left( TM\right) ,\text{ \ \ }u\in \Gamma \left(
TM^{\perp }\right) .
\end{equation*}%
Thus, the 1-dimensional $RadTM$\ of a half-lightlike submanifold $M$\ is
locally spanned by $\xi $. In this case there exists a supplementary
distribution $S(TM)$\ to $RadTM$\ in $TM.$\ Next, consider the orthogonal
complementary distribution $S(TM^{\perp })$\ to $S(TM)$\ in $T\bar{M}.$\
Certainly $\xi $\ and $u$\ belong to $\Gamma \left( S(TM)^{\perp }\right) . $%
\ From now on, we choose $u$\ as a unit vector field and put%
\begin{equation*}
\bar{g}\left( u,u\right) =\epsilon
\end{equation*}%
where $\epsilon =\pm 1.$\ Since $RadTM$\ is a 1-dimensional vector
sub-bundle of $TM^{\perp }$\ we may consider a supplementary distribution $D$%
\ to $RadTM$\ such that it is locally represented by $u.$\ We call $D$\ a
screen transversal bundle of $M$. Hence we have the orthogonal decomposition%
\begin{equation*}
S(TM)^{\perp }=D\perp D^{\perp },
\end{equation*}%
where $D^{\perp }$\ is orthogonal complementary distribution to $D$\ in $%
S(TM)^{\perp }.$ Taking into account that $D^{\perp }$\ is non-degenerate
and $\xi \in \Gamma \left( D^{\perp }\right) $, there exists a unique
locally defined vector field $N\in \Gamma \left( D^{\perp }\right) $,
satisfying%
\begin{equation}
\bar{g}\left( N,\xi \right) \neq 0,\text{ \ \ }\bar{g}\left( N,N\right) =%
\bar{g}\left( N,u\right) =0  \label{1.1}
\end{equation}%
if and only if \ $N$\ is given by%
\begin{equation}
N=\frac{1}{\bar{g}\left( V,\xi \right) }\left\{ V-\frac{\bar{g}\left(
V,V\right) }{2\bar{g}\left( V,\xi \right) }\xi \right\} ,\text{ \ \ }V\in
\Gamma \left( F_{\mid U}\right)  \label{1.2}
\end{equation}%
such that $\bar{g}\left( \xi ,V\right) \neq 0.$\ Here, $F$\ is a
complementary vector bundle of $RadTM$\ in $D^{\perp }$. Hence $N$\ is a
lightlike vector field which is neither tangent to $M$\ nor collinear with $%
u $\ since $\bar{g}\left( u,\xi \right) =0$. If we choose $\xi ^{\ast
}=\alpha \xi $\ on another neighborhood of coordinates, then we obtain $%
N^{\ast }=\frac{1}{\alpha }N$. Thus we say that the vector bundle $tr(TM)$\
defined over $M$\ by%
\begin{equation*}
tr(TM)=D\oplus _{orth}Itr(TM),
\end{equation*}%
where $Itr(TM)$\ is a 1-dimensional vector bundle locally represented by $N$%
, is the lightlike transversal bundle of $M$\ with respect to the screen
distribution $S(TM)$. Therefore,%
\begin{eqnarray}
T\bar{M} &=&S(TM)\perp \left( RadTM\oplus tr\left( TM\right) \right)  \notag
\\
&=&S(TM)\perp D\perp \left( RadTM\oplus Itr(TM)\right)  \label{1.3}
\end{eqnarray}%
as per decomposition (\ref{1.3}), choose the field of frames $\left\{ \xi
,F_{1},...,F_{m-1}\right\} $\ \ and $\left\{ \xi
,F_{1},...,F_{m-1},u,N\right\} $\ on $M$\ and $\bar{M}$\ respectively, where
$\left\{ F_{1},...,F_{m-1}\right\} $\ is an orthonormal basis of $\Gamma
\left( S(TM)\right)$ \cite{Krishan L. Duggal and Bayram Sahin}.
\end{definition}

Denote by $P$\ the projection of $TM$\ on $S(TM)$\ with respect to the
decomposition (\ref{1.3}) then we write%
\begin{equation*}
X=PX+\eta \left( X\right) \xi ,\text{ \ \ }\forall X\in \Gamma \left(
TM\right) ,
\end{equation*}%
where $\eta $\ is a local differential 1-form on $M$ defined by $\eta \left(
X\right) =g\left( X,N\right) .$ Suppose $\bar{\nabla}$\ is the metric
connection on $\bar{M}.$\ Since $\left\{ \xi ,N\right\} $\ is locally a pair
of lightlike sections on $U\subset M$, we define symmetric $F\left( M\right)
$-bilinear forms $D_{1}$\ and $D_{2}$\ and 1-forms $\rho _{1},\rho
_{2},\varepsilon _{1}$\ and $\varepsilon _{2}$\ on $U.$\ Using (\ref{1.3}),
we put%
\begin{eqnarray}
\bar{\nabla}_{X}Y &=&\nabla _{X}Y+D_{1}\left( X,Y\right) N+D_{2}\left(
X,Y\right) u  \label{1.4} \\
\bar{\nabla}_{X}N &=&-A_{N}X+\rho _{1}\left( X\right) N+\rho _{2}\left(
X\right) u\text{ \ }  \label{1.5} \\
\bar{\nabla}_{X}u &=&-A_{u}X+\varepsilon _{1}\left( X\right) N+\varepsilon
_{2}\left( X\right) u  \label{1.6}
\end{eqnarray}%
for any $X,Y\in \Gamma \left( TM\right) ,$\ where $\nabla _{X}Y,$ $A_{N}X$\
and $A_{u}X$\ belong to $\Gamma \left( TM\right) $. We called $D_{1}$\ and $%
D_{2}$\ the lightlike second fundamental form and screen second fundamental
form of $M$\ with respect to $tr(TM)$\ respectively. Both $A_{N}$\ and $%
A_{u} $\ are linear operators on $\Gamma \left( TM\right) $. The first one
is $\Gamma \left( S\left( TM\right) \right) $-valued, called the shape
operator of $M$. Since $u$\ is a unit vector field, (\ref{1.6}) implies $%
\varepsilon _{2}\left( X\right) =0.$\ In a similar way, since $\xi $\ and $N$%
\ are lightlike vector fields, from (\ref{1.4})-(\ref{1.6}) we obtain%
\begin{eqnarray}
D_{1}\left( X,\xi \right) &=&0,\text{ }  \label{1.7} \\
\bar{g}\left( A_{N}X,N\right) &=&0,  \label{1.8} \\
\bar{g}\left( A_{u}X,Y\right) &=&\epsilon D_{2}\left( X,Y\right)
+\varepsilon _{1}\left( X\right) \eta \left( Y\right)  \label{1.9} \\
\varepsilon _{1}\left( X\right) &=&-\epsilon D_{2}\left( X,\xi \right) ,%
\text{ \ \ }\forall X\in \Gamma \left( TM\right)  \label{1.10}
\end{eqnarray}%
Next, consider the decomposition (\ref{1.3}) then we have%
\begin{eqnarray}
\nabla _{X}PY &=&\nabla _{X}^{\ast }PY+E_{1}\left( X,PY\right) \xi
\label{1.11} \\
\nabla _{X}\xi &=&-A_{\xi }^{\ast }X+u_{1}\left( X\right) \xi  \label{1.12}
\end{eqnarray}%
where $\nabla _{X}^{\ast }PY$\ and $A_{\xi }^{\ast }$\ belong to $\Gamma
\left( S(TM)\right) .$\ $A_{\xi }^{\ast }$\ is a linear operator on $\Gamma
\left( TM\right) $\ and $\nabla ^{\ast }$\ is a metric connection on $S(TM)$%
. We call $E_{1}$\ the local second fundamental form of $S(TM)$\ with to
respect to $Rad(TM)$\ and $A_{\xi }^{\ast }$\ the shape operator of the
screen distribution. The geometric object from Gauss and Weingarten
equations (\ref{1.4})-(\ref{1.6}) on one side and (\ref{1.11} )-(\ref{1.12})
on the other side are related by%
\begin{eqnarray}
E_{1}\left( X,PY\right) &=&g\left( A_{N}X,PY\right) ,  \label{1.13} \\
D_{1}\left( X,PY\right) &=&g\left( A_{\xi }^{\ast }X,PY\right) ,
\label{1.14} \\
u_{1}\left( X\right) &=&-\rho _{1}\left( X\right) ,  \notag
\end{eqnarray}%
for any $X,Y\in \Gamma \left( TM\right) $. From (\ref{1.7}) and (\ref{1.14})
we derive%
\begin{equation}
A_{\xi }^{\ast }\xi =0.  \label{1.15}
\end{equation}%
\ A half-lightlike submanifold $(M,g)$\ of a semi-Riemannian manifold $(\bar{%
M},\bar{g})$\ is said to be totally umbilical in $\bar{M}$\ if there is a
normal vector field $\acute{Z}\in \Gamma \left( tr\left( TM\right) \right) $%
\ on $M$, called an affine normal curvature vector field of $M$, such that%
\begin{equation*}
h(X,Y)=D_{1}\left( X,Y\right) N+D_{2}\left( X,Y\right) u=\acute{Z}\bar{g}%
\left( X,Y\right) ,\text{ \ \ }\forall X,Y\in \Gamma \left( TM\right) .
\end{equation*}%
In particular, $(M,g)$\ is said to be totally geodesic if its second
fundamental form $h(X,Y)=0$\ for any $X,Y\in \Gamma \left( TM\right) $. By
direct calculation it is easy to see that $M$\ is totally geodesic if and
only if both the lightlike and the screen second fundamental tensors $D_{1}$%
\ and $D_{2}$\ respectively vanish on $M.$\ Moreover, from (\ref{1.5}), (\ref%
{1.9}), (\ref{1.10}) and (\ref{1.14}) we obtain%
\begin{equation*}
A_{\xi }=A_{u}=\varepsilon _{1}=\rho _{2}=0.
\end{equation*}%
\ The notion of screen locally conformal half-lightlike submanifolds has
been introduced by Duggal-Sahin \cite{Krishan L. Duggal and Bayram Sahin} as
follows.

\begin{definition}
A half-lightlike submanifold $M$, of a semi-Riemannian manifold, is called
screen locally conformal if on any coordinate neighborhood $U$\ there exists
a non-zero smooth function $\varphi $\ such that for any null vector field $%
\xi \in \Gamma \left( TM^{\perp }\right) $\ the relation%
\begin{equation}
A_{N}X=\varphi A_{\xi }^{\ast }X\text{, \ \ }\forall X\in \Gamma \left(
TM_{\mid U}\right)  \label{1.16}
\end{equation}%
holds between the shape operators $A_{N}$\ and $A_{\xi }^{\ast }$\ of $M$\
and $S(TM)$\ respectively\cite{Krishan L. Duggal and Bayram Sahin}.
\end{definition}

On the other hand the notion of minimal lightlike submanifolds has been
defined by Bejancu-Duggal as follows.

\begin{definition}
Let $M$\ be a half-lightlike submanifold of a semi-Riemannian manifold $\bar{%
M}.$\ Then, we say that $M$\ is a minimal half-lightlike submanifold if $%
\left( tr\mid _{S(TM)}h=0\right) $\ and $\varepsilon _{1}\left( X\right) =0$.
\end{definition}

\begin{definition}
\ A half-lightlike submanifold $M$\ is said to be irrotational if \ $\bar{%
\nabla }_{X}\xi \in \Gamma \left( TM\right) $\ for any $X\in \Gamma \left(
TM\right) $, where $\xi \in \Gamma \left( RadTM\right) $\cite{Krishan L.
Duggal and Bayram Sahin}.
\end{definition}

For a half-lightlike $M$, since $D_{1}\left( X,\xi \right) =0$, the above
definition is equivalent to $D_{2}\left( X,\xi \right) =0=\varepsilon
_{1}\left( X\right) ,$\ $\forall X\in \Gamma \left( TM\right) $.

\begin{corollary}
Let $M$\ be an irrotational screen conformal half-lightlike submanifold of a
semi-Riemannian manifold $\bar{M}$. Then

1. $M$\ is totally geodesic,

2. $M$\ is totally umbilical,

3. $M$\ is minimal,

if and only if a leaf $M^{\prime }$\ of any $S(TM)$\ is so immersed as a
submanifold of \ $\bar{M}$\cite{Krishan L. Duggal and Bayram Sahin}.
\end{corollary}

\bigskip

\section{Planar normal sections \ of half-lightlike hypersurfaces in $%
R_{2}^{4}$}

In this section we consider half-lightlike submanifolds having planar normal
section. First, we consider degenerate planar normal sections.

\subsection{Degenerate Planar Normal Section in Half-Lightlike Hypersurfaces}

\qquad \bigskip

\qquad Let $M$\ be a half-lightlike hypersurface in\ $R_{2}^{4}.$\ Now we
investigate the conditions for a half-lightlike hypersurface of $R_{2}^{4}$\
to have degenerate planar normal sections.

\begin{theorem}
\label{th.3.1.}\ Let $M$\ be a half-lightlike hypersurface in $R_{2}^{4}.$\
Then $M$\ has planar normal sections if and only if%
\begin{equation}
D_{2}\left( \xi ,\xi \right) u\wedge \bar{\nabla }_{\xi }D_{2}\left( \xi
,\xi \right) u=0  \label{2.1.5}
\end{equation}%
where $D_{2}$\ is the screen second fundamental form of $M.$
\end{theorem}

\begin{proof}
If $\gamma $\ is a null curve, for a point $p$\ in $M,$\ we have%
\begin{eqnarray}
\gamma ^{\prime }\left( s\right) &=&\xi \text{,\ \ \ }  \label{2.1.1} \\
\gamma ^{\prime \prime }\left( s\right) &=&\bar{\nabla}_{\xi }\xi =\nabla
_{\xi }\xi +D_{2}\left( \xi ,\xi \right) u,  \label{2.1.2} \\
\gamma ^{\prime \prime \prime }\left( s\right) &=&\nabla _{\xi }\nabla _{\xi
}\xi +D_{2}\left( \nabla _{\xi }\xi ,\xi \right) u  \label{2.1.3} \\
&&+\xi \left( D_{2}\left( \xi ,\xi \right) \right) u+D_{2}\left( \xi ,\xi
\right) \left( -A_{u}\xi +\varepsilon _{1}\left( \xi \right) N\right) .
\notag
\end{eqnarray}%
From the definition of planar normal section and using $Rad(TM)=Sp\left\{
\xi \right\} ,$\ we get%
\begin{equation}
\nabla _{\xi }\xi \wedge \xi =0\text{ {\large and} }\nabla _{\xi }\nabla
_{\xi }\xi \wedge \xi =0.  \label{2.1.4}
\end{equation}%
Assume that $M$\ has planar degenerate normal sections. Then
\begin{equation}
\gamma ^{\prime \prime \prime }\left( s\right) \wedge \gamma ^{\prime \prime
}\left( s\right) \wedge \gamma ^{\prime }\left( s\right) =0.  \label{2.1.4i}
\end{equation}%
Thus, by using (\ref{2.1.1})-(\ref{2.1.4}) in (\ref{2.1.4i}) \ one can see
that $D_{2}\left( \xi ,\xi \right) u$\ and $D_{2}\left( \nabla _{\xi }\xi
,\xi \right) u+\xi \left( D_{2}\left( \xi ,\xi \right) \right) u-D_{2}\left(
\xi ,\xi \right) A_{u}\xi +D_{2}\left( \xi ,\xi \right) \varepsilon
_{1}\left( \xi \right) N$\ are linearly dependent. Taking covariant
derivative of $D_{2}\left( \xi ,\xi \right) u$ we obtain\
\begin{equation*}
\bar{\nabla}_{\xi }(D_{2}\left( \xi ,\xi \right) u)=\xi \left( D_{2}\left(
\xi ,\xi \right) \right) u-D_{2}\left( \xi ,\xi \right) A_{u}\xi
+D_{2}\left( \xi ,\xi \right) \varepsilon _{1}\left( \xi \right) N
\end{equation*}%
where $\gamma $\ is assumed to be parameterized by distinguished parameter.
Hence we get%
\begin{equation*}
D_{2}\left( \xi ,\xi \right) u\wedge \bar{\nabla}_{\xi }D_{2}\left( \xi ,\xi
\right) u=0.
\end{equation*}%
Conversely, assume that $D_{2}\left( \xi ,\xi \right) u\wedge \bar{\nabla}%
_{\xi }D_{2}\left( \xi ,\xi \right) u=0$\ for degenerate tangent vector $\xi
$\ of $M$\ at $p.$ In this case, either $D_{2}\left( \xi ,\xi \right) u=0$\
or $\bar{\nabla}_{\xi }D_{2}\left( \xi ,\xi \right) u=0.$ If $\ D_{2}\left(
\xi ,\xi \right) u=0,$\ then $M$\ is totally geodesic in $\bar{M}$\ and $M$\
is totally umbilical. Thus, we obtain
\begin{eqnarray}
\gamma ^{\prime }\left( s\right) &=&\xi \text{ ,}  \notag \\
\gamma ^{\prime \prime }\left( s\right) &=&u_{1}\left( \xi \right) \xi ,
\label{2.1.4ii} \\
\gamma ^{\prime \prime \prime }\left( s\right) &=&\nabla _{\xi }\nabla _{\xi
}\xi =\xi \left( u_{1}\left( \xi \right) \xi \right) +u_{1}^{2}\left( \xi
\right) \xi .  \notag
\end{eqnarray}%
which give that $M$\ \ has degenerate planar normal sections. On the other
hand if $\bar{\nabla}_{\xi }D_{2}\left( \xi ,\xi \right) u=0,$\ then $M$\ is
screen conformal. Hence we have
\begin{equation*}
\gamma ^{\prime \prime \prime }\left( s\right) \wedge \gamma ^{\prime \prime
}\left( s\right) \wedge \gamma ^{\prime }\left( s\right) =\xi \wedge
D_{2}\left( \xi ,\xi \right) u\wedge \bar{\nabla}_{\xi }D_{2}\left( \xi ,\xi
\right) u=0.
\end{equation*}
\end{proof}

\bigskip

Now we define a function
\begin{eqnarray*}
L_{p}:RadT_{p}M &\rightarrow &R, \\
\xi &\rightarrow &L_{p}\left( \xi \right) =D_{2}^{2}\left( \xi ,\xi \right)
\epsilon
\end{eqnarray*}%
where\ $p\in M$\ and $\gamma \left( 0\right) =p.$\ If \ $L_{p}\left( \xi
\right) =D_{2}^{2}\left( \xi ,\xi \right) \epsilon =0$, then we obtain $%
D_{2}\left( \xi ,\xi \right) =0$\ and $\varepsilon _{1}\left( \xi \right)
=0. $\ From (\ref{2.1.4ii}) we find $\gamma ^{\prime \prime \prime }\left(
s\right) \wedge \gamma ^{\prime \prime }\left( s\right) \wedge \gamma
^{\prime }\left( s\right) =0.$\ Hence, $M$\ has degenerate planar normal
sections.

We say that the curve $\gamma $\ has a \ vertex at the point $p$ if the
curvature $\kappa $\ of $\gamma $ satisfies $\frac{d\kappa ^{2}\left(
p\right) }{ds}=0$ and\ $\kappa ^{2}=\left\langle \gamma ^{\prime \prime
}\left( s\right) ,\gamma ^{\prime \prime }\left( s\right) \right\rangle .$\
Now let $M$\ has degenerate planar normal sections. Then $L_{p}=0,$\ and so $%
D_{2}\left( \xi ,\xi \right) =0.$\ Hence, we get%
\begin{eqnarray*}
h\left( \xi ,\xi \right) &=&D_{2}(\xi ,\xi )u=0, \\
\left( \bar{\nabla }_{\xi }h\right) \left( \xi ,\xi \right) &=&0
\end{eqnarray*}%
which give $\bar{\nabla }h=0.$\ Moreover, we have%
\begin{equation*}
\epsilon \kappa ^{2}\left( s\right) =\left\langle \gamma ^{\prime \prime
}\left( s\right) ,\gamma ^{\prime \prime }\left( s\right) \right\rangle =0
\end{equation*}%
for any $p\in M.$

Consequently, we have the following result.

\begin{theorem}
Let $M$\ be a half-lightlike hypersurface in $R_{2}^{4}$\ with degenerate
planar normal sections such that
\begin{eqnarray*}
L_{p}:RadT_{p}M &\rightarrow &R, \\
\xi &\rightarrow &L_{p}\left( \xi \right) =D_{2}^{2}\left( \xi ,\xi \right)
\epsilon
\end{eqnarray*}%
where\ $p\in M.$\ Then the following statements are equivalent

$\left( 1\right) $ $D_{2}\left( \xi ,\xi \right) =0,$

$\left( 2\right) $\ $\left( \bar{\nabla }_{\xi }h\right) \left( \xi ,\xi
\right) =0,$

$\left( 3\right) $\ $\bar{\nabla }h=0,$

$\left( 4\right) $\ For any $p\in M,$\ $\kappa =0.$
\end{theorem}

\bigskip

Now, let a half-lightlike hypersurface $M$ of $R_{2}^{4}$\ has degenerate
planar normal sections.\ Then for null vector $\xi \in RadTM,$\ we have%
\begin{equation}
\nabla _{\xi }\xi \neq 0  \label{2.1.9}
\end{equation}%
where $\xi =\gamma ^{\prime }\left( s\right) $, namely, the normal section $%
\gamma $\ is not a geodesic arc on a sufficiently small neighborhood of $p$.
Then from (\ref{2.1.1})-(\ref{2.1.3}) we write

\begin{equation*}
\gamma ^{\prime \prime \prime }\left( s\right) =a\left( s\right) \gamma
^{\prime \prime }\left( s\right) +b\left( s\right) \gamma ^{\prime }\left(
s\right) .
\end{equation*}

where, $a$ and $b$\ are differentiable functions for all $p\in M.$ Hence, we
get $D_{2}\left( \xi ,\xi \right) =\varepsilon _{1}\left( \xi \right) =0$.

Consequently, we have the following

\begin{theorem}
Let a half-lightlike hypersurface $M$ in $R_{2}^{4}$\ has degenerate planar
normal sections.\ If the normal section $\gamma $\ at for any $p$\ is not a
geodesic arc on a sufficiently small neighborhood of $p,$\ then $D_{2}=0$\
at $RadTM.$
\end{theorem}

\bigskip

Next, assume that $\gamma $\ is parameterized by distinguish, namely, $%
\gamma $\ is a geodesic arc on a small neighborhood of $p=\gamma \left(
0\right) $, i.e., $\nabla _{\xi }\xi =0$. Thus, from $u_{1}\left( \xi
\right) =\rho _{1}\left( \xi \right) =0,$\ we obtain%
\begin{eqnarray}
\gamma ^{\prime }\left( 0\right) &=&\xi ,  \notag \\
\gamma ^{\prime \prime }\left( 0\right) &=&D_{2}(\xi ,\xi )u\text{ ,}
\label{2.1.10} \\
\gamma ^{\prime \prime \prime }\left( 0\right) &=&\xi \left( D_{2}\left( \xi
,\xi \right) \right) u-D_{2}\left( \xi ,\xi \right) A_{u}\xi -\epsilon
D_{2}^{2}\left( \xi ,\xi \right) N,  \notag \\
\bar{\nabla}_{\xi }D_{2}\left( \xi ,\xi \right) u &=&\xi \left( D_{2}\left(
\xi ,\xi \right) \right) u-D_{2}\left( \xi ,\xi \right) A_{u}\xi -\epsilon
D_{2}^{2}\left( \xi ,\xi \right) N.\text{ }  \notag
\end{eqnarray}%
Now, let suppose that $M$\ has degenerate planar normal sections at $\gamma
\left( 0\right) =p.$\ Since $\gamma $\ lies in plane through $p$\ spanned by
$\xi $\ and $\left\{ N,u\right\} $, we write%
\begin{equation}
\gamma \left( s\right) =p+a\left( s\right) \xi +b\left( s\right) N+c\left(
s\right) u  \label{2.1.11}
\end{equation}%
for some functions $a,b$\ and $c$\ . Thus, from (\ref{2.1.10}) and (\ref%
{2.1.11}), we obtain%
\begin{equation}
\gamma ^{\prime \prime \prime }\left( s\right) =a^{\prime \prime \prime
}\left( s\right) \xi +b^{\prime \prime \prime }\left( s\right) N+c^{\prime
\prime \prime }\left( s\right) u=\bar{\nabla}_{\xi }D_{2}\left( \xi ,\xi
\right) u.  \label{2.1.12}
\end{equation}%
We calculate%
\begin{eqnarray}
\left\langle h\left( \xi ,\xi \right) ,h\left( \xi ,w\right) \right\rangle
&=&\left\langle h\left( \xi ,\xi \right) ,\bar{\nabla}_{w}\xi \right\rangle
-\left\langle h\left( \xi ,\xi \right) ,\nabla _{w}\xi \right\rangle  \notag
\\
&=&\epsilon D_{2}\left( \xi ,\xi \right) D_{2}\left( w,\xi \right) .
\label{2.1.13}
\end{eqnarray}%
From the symmetry of bilinear forms $D_{1}$ and $D_{2}$\ at $\Gamma \left(
TM\right) ,$ we obtain%
\begin{eqnarray}
\left\langle h\left( \xi ,\xi \right) ,h\left( \xi ,w\right) \right\rangle
&=&\left\langle h\left( \xi ,\xi \right) ,\bar{\nabla}_{\xi }w\right\rangle
-\left\langle h\left( \xi ,\xi \right) ,\nabla _{\xi }w\right\rangle  \notag
\\
&=&-\left\langle a^{\prime \prime \prime }\left( s\right) \xi +b^{\prime
\prime \prime }\left( s\right) N+c^{\prime \prime \prime }\left( s\right)
u,w\right\rangle  \notag \\
&=&0.  \label{2.1.14}
\end{eqnarray}%
Thus, from (\ref{2.1.13}) and (\ref{2.1.14}), we get $D_{2}=0$\ at $\Gamma
\left( TM\right) .$\ Furthermore, from $\bar{\nabla}_{w}\xi \in $\ $\Gamma
\left( TM\right) ,$ $(\xi \in RadTM$\ and $w\in $\ $\Gamma \left( TM\right)
),$\ we see that $M$\ is irrotational.

Then we have the following result,

\begin{theorem}
Let $M$\ be a half-lightlike hypersurface of $R_{2}^{4}$\ with degenerate
planar normal sections.\ If the normal section $\gamma $\ at for any $p$\ is
a geodesic arc on a sufficiently small neighborhood of $p,$\ then $M$\ is
irrotational.
\end{theorem}

\bigskip

Let $M$\ be a half-lightlike hypersurface in $R_{2}^{4}$\ with degenerate
planar normal sections.\ Since $\gamma $\ is a planar curve we write

\begin{equation*}
\gamma ^{\prime \prime \prime }\left( s\right) =a\left( s\right) \gamma
^{\prime \prime }\left( s\right) +b\left( s\right) \gamma ^{\prime }\left(
s\right) .
\end{equation*}

where, $a$ and $b$\ are differentiable functions for all $p\in M.$Then (\ref%
{2.1.4ii}) gives%
\begin{eqnarray*}
a\left( s\right) &=&u_{1}\left( \xi \right) +\xi \left( \ln \left(
D_{2}\left( \xi ,\xi \right) \right) \right) , \\
b\left( s\right) &=&\xi \left( u_{1}\left( \xi \right) \right) -D_{2}\left(
\xi ,\xi \right) \rho _{2}\left( \xi \right) \epsilon -u_{1}\left( \xi
\right) \xi \left( \ln \left( D_{2}\left( \xi ,\xi \right) \right) \right) .
\end{eqnarray*}%
\ Moreover, we have $\epsilon \kappa ^{2}\left( s\right) =\left\langle
\gamma ^{\prime \prime }\left( s\right) ,\gamma ^{\prime \prime }\left(
s\right) \right\rangle =0$\ for any $p\in M$ \ which gives\ $D_{2}\left( \xi
,\xi \right) =\varepsilon _{1}\left( \xi \right) =0$. Thus, we obtain%
\begin{eqnarray}
\gamma ^{\prime \prime \prime }\left( s\right) &=&u_{1}^{2}\left( \xi
\right) \xi +u_{1}\left( \xi \right) D_{2}(\xi ,\xi )u  \notag \\
&&+\xi \left( \ln \left( D_{2}\left( \xi ,\xi \right) \right) \right)
D_{2}\left( \xi ,\xi \right) u  \label{2.1.15} \\
&&+\xi \left( u_{1}\left( \xi \right) \right) \xi -\epsilon D_{2}(\xi ,\xi
)\rho _{2}\left( \xi \right) \xi  \notag
\end{eqnarray}%
and%
\begin{equation}
A_{u}\xi =\epsilon \rho _{2}\left( \xi \right) \xi .  \label{2.1.16}
\end{equation}%
Namely,

\begin{corollary}
\label{co.3.1}\bigskip\ Let $M$\ be a half-lightlike hypersurface of $%
R_{2}^{4}$\ with degenerate planar normal sections$,$then $A_{u}\xi $\ is $%
RadTM$-valued
\end{corollary}

Now, from (\ref{2.1.15}) and (\ref{2.1.16}), we obtain%
\begin{eqnarray}
\left( \bar{\nabla }_{\xi }h\right) \left( \xi ,\xi \right) &=&\xi \left(
\ln \left( D_{2}\left( \xi ,\xi \right) \right) \right) D_{2}\left( \xi ,\xi
\right) u  \notag \\
&&-\epsilon D_{2}(\xi ,\xi )\rho _{2}\left( \xi \right) \xi -2u_{1}\left(
\xi \right) D_{2}(\xi ,\xi )u  \label{2.1.17}
\end{eqnarray}%
Let $M$\ be a half-lightlike hypersurface of $R_{2}^{4}$\ with degenerate
planar normal sections. If the normal section $\gamma $\ at for any $p$\ is
not a geodesic arc on a sufficiently small neighborhood of $p,$\ then we
obtain%
\begin{equation}
D_{2}(\xi ,\xi )u\wedge \left( \bar{\nabla }_{\xi }h\right) \left( \xi ,\xi
\right) =0.  \label{2.1.18}
\end{equation}%
Conversely, we assume that the eq. (\ref{2.1.18}) is satisfied\ for any
degenerate tangent vector $\xi $\ of $M.$\ Then either $D_{2}(\xi ,\xi )u=0$%
\ or $\left( \bar{\nabla }_{\xi }h\right) \left( \xi ,\xi \right) =0$. If $%
D_{2}(\xi ,\xi )u=0,$\ then from Theorem \ref{th.3.1.}, we see that $M$\ has
degenerate planar normal sections. On the other hand, if $\left( \bar{\nabla
}_{\xi }h\right) \left( \xi ,\xi \right) =0,$\ then, by considering (\ref%
{2.1.4}), we obtain%
\begin{equation*}
\gamma ^{\prime \prime \prime }\left( s\right) \wedge \gamma ^{\prime \prime
}\left( s\right) \wedge \gamma ^{\prime }\left( s\right) =\xi \wedge
D_{2}(\xi ,\xi )u\wedge \left( \bar{\nabla }_{\xi }h\right) \left( \xi ,\xi
\right) =0.
\end{equation*}%
Consequently, we have the following,

\begin{theorem}
\ Let $M$\ be half-lightlike hypersurface of $R_{2}^{4}$ such that the
normal section $\gamma \left( s\right) $\ at for any $p$\ is not a geodesic
arc on a sufficiently small neighborhood of $p$. Then half-lightlike
hypersurface $M$\ has planar normal sections if and only if (\ref{2.1.18})
is satisfied.
\end{theorem}

\bigskip

Now, let\ the normal section $\gamma $\ is a geodesic arc on a sufficiently
small neighborhood of $p,$\ namely, $\nabla _{\xi }\xi =0=u_{1}\left( \xi
\right) .$\ Since $M$\ has degenerate planar normal sections, we obtain%
\begin{equation*}
\gamma ^{\prime \prime \prime }\left( s\right) \wedge \gamma ^{\prime \prime
}\left( s\right) \wedge \gamma ^{\prime }\left( s\right) =(\xi \wedge
D_{2}(\xi ,\xi )u\wedge D_{2}\left( \xi ,\xi \right) A_{u}\xi )+(\xi \wedge
D_{2}(\xi ,\xi )u\wedge D_{2}\left( \xi ,\xi \right) \varepsilon _{1}\left(
\xi \right) N).
\end{equation*}%
From corollary\ref{co.3.1}, we have $D_{2}\left( \xi ,\xi \right) =0$\ and $%
\varepsilon _{1}\left( \xi \right) =0.$

\bigskip Thus we have the following result,

\begin{theorem}
\ Let $M$\ has degenerate planar normal sections half-lightlike hypersurface
of $R_{2}^{4}.$ The normal section $\gamma $\ at for any $p$\ is a geodesic
arc on a sufficiently small neighborhood of $p.$\ Then $D_{2}(\xi ,\xi )=0$\
or $\varepsilon _{1}\left( \xi \right) =0.$
\end{theorem}

\bigskip

Let $M$\ be a screen conformal half-lightlike hypersurface of $R_{2}^{4}(c)$
with\ degenerate planar normal sections. We denote Riemann curvature tensor
of $\bar{M}$ and $M$ by $\bar{R}$ and $R$, the following formula is well
known,
\begin{eqnarray}
\bar{R}\left( X,Y\right) Z &=&R(X,Y)Z+D_{1}(X,Z)A_{N}Y  \notag \\
&&-D_{1}(Y,Z)A_{N}X+D_{2}(X,Z)A_{U}Y-D_{2}(Y,Z)A_{U}X  \notag \\
&&+\{(\nabla _{X}D_{1})(Y,Z)-\left( \nabla _{Y}D_{1}\right) \left( X,Z\right)
\notag \\
&&+\rho _{1}\left( X\right) D_{1}(Y,Z)-\rho _{1}\left( Y\right) D_{1}(X,Z)
\label{2.1.19i} \\
&&+\varepsilon _{1}\left( X\right) D_{2}(Y,Z)-\varepsilon _{1}\left(
Y\right) D_{2}(X,Z)\}N  \notag \\
&&+\{(\nabla _{X}D_{2})(Y,Z)-\left( \nabla _{Y}D_{2}\right) \left( X,Z\right)
\notag \\
&&+\rho _{2}\left( X\right) D_{1}(Y,Z)-\rho _{2}\left( Y\right)
D_{1}(X,Z)\}u.  \notag
\end{eqnarray}%
Hence we have%
\begin{eqnarray}
\bar{g}\left( \bar{R}\left( X,Y\right) Z,PW\right) &=&\varphi \left[
D_{1}(X,Z)D_{1}(Y,PW)-D_{1}(Y,Z)D_{1}(X,PW)\right]  \notag \\
&&+\epsilon \left[ D_{2}(X,Z)D_{2}(Y,PW)-D_{2}(Y,Z)D_{2}(X,PW)\right] .
\label{2.1.19}
\end{eqnarray}%
Let $p\in M$\ and $\xi $\ be a null vector of $T_{p}M$. A plane $H$\ of $%
T_{p}M$\ is called a null plane directed by $\xi ,$\ if it contains $\xi $, $%
\bar{g}(\xi ,W)=0$\ for any $W\in H$\ and there exits $W_{0}\in H$\ such
that $\bar{g}(W_{0},W_{0})\neq 0$. Then the null sectional curvature of $H$\
with respect to $\xi $\ and $\bar{\nabla }$\ is defined by%
\begin{equation}
K_{\xi }\left( H\right) =\frac{R_{p}\left( W,\xi ,\xi ,W\right) }{%
g_{p}\left( W,W\right) }.  \label{2.1.20}
\end{equation}%
Since $v\in \Gamma \left( S\left( TM\right) \right) $\ and $\xi \in \Gamma
\left( RadTM\right) ,$\ we have%
\begin{eqnarray*}
K_{\xi }\left( H\right) &=&\varphi \left[ D_{1}(v,\xi )D_{1}(\xi
,v)-D_{1}(\xi ,\xi )D_{1}(v,v)\right] \\
&&+\epsilon \left[ D_{2}(v,\xi )D_{2}(\xi ,v)-D_{2}(\xi ,\xi )D_{2}(v,v)%
\right] ,
\end{eqnarray*}%
By using $D_{1}(v,\xi )=0$ in the last equation we obtain%
\begin{equation}
K_{\xi }\left( H\right) =\epsilon \left[ D_{2}(v,\xi )D_{2}(\xi
,v)-D_{2}(\xi ,\xi )D_{2}(v,v)\right] .  \label{2.1.21}
\end{equation}

Consequently, we have the following

\begin{theorem}
\label{th.3.7}\ Let $M$\ be a screen conformal half-lightlike hypersurface
of $R_{2}^{4}(c)$ with degenerate planar normal sections.\ If \ $M$\ is
minimal, then $K_{\xi }\left( H\right) =0.$
\end{theorem}

\begin{example}
\label{ex1}Consider a surface $M$\ in $R_{2}^{4}$\ given by the equation%
\begin{equation*}
x^{3}=\frac{1}{\sqrt{2}}\left( x^{1}+x^{2}\right) ;\text{ \ \ }x^{4}=\frac{1%
}{2}\log \left( 1+\left( x^{1}-x^{2}\right) ^{2}\right) .
\end{equation*}%
Then $TM=Span\left\{ U_{1},U_{2}\right\} $\ and $TM^{\perp }=Span\left\{ \xi
,u\right\} $\ where%
\begin{eqnarray*}
U_{1} &=&\sqrt{2}\left( 1+\left( x^{1}-x^{2}\right) ^{2}\right) \partial
_{1}+\left( 1+\left( x^{1}-x^{2}\right) ^{2}\right) \partial _{3}+\sqrt{2}%
\left( x^{1}-x^{2}\right) \partial _{4}, \\
U_{2} &=&\sqrt{2}\left( 1+\left( x^{1}-x^{2}\right) ^{2}\right) \partial
_{1}+\left( 1+\left( x^{1}-x^{2}\right) ^{2}\right) \partial _{3}-\sqrt{2}%
\left( x^{1}-x^{2}\right) \partial _{4}, \\
\xi &=&\partial _{1}+\partial _{2}+\sqrt{2}\partial _{3}\text{ } \\
u &=&2\left( x^{2}-x^{1}\right) \partial _{2}+\sqrt{2}\left(
x^{2}-x^{1}\right) \partial _{3}+\left( 1+\left( x^{1}-x^{2}\right) \right)
\partial _{4}.
\end{eqnarray*}%
By straightforward calculations we can see that $RadTM$\ is a distribution
on $M$\ of rank 1 spanned by $\xi $. Hence $M$\ is a half-lightlike
submanifold of $R_{2}^{4}$. Choose $S(TM)$\ and $D$\ spanned by $U_{2}$\ and
$u,$\ respectively\ where $U_{2}$ is\ timelike and $u$ is spacelike. We
obtain the null canonical affine normal bundle%
\begin{equation*}
Itr(TM)=span\left\{ N=-\frac{1}{2}\partial _{1}+\frac{1}{2}\partial _{2}+%
\frac{1}{\sqrt{2}}\partial _{3}\text{\ }\right\} ,
\end{equation*}%
and the canonical affine normal bundle $tr(TM)=Span\left\{ N,u\right\} .$\
Then by straightforward calculations we obtain%
\begin{eqnarray*}
\bar{\nabla}_{U_{2}}U_{2} &=&2\left( 1+\left( x^{1}-x^{2}\right) ^{2}\right)
.\left\{ 2\left( x^{2}-x^{1}\right) \partial _{2}+\sqrt{2}\left(
x^{2}-x^{1}\right) \partial _{3}+\partial _{4}\right\} , \\
\bar{\nabla}_{\xi }U_{2} &=&0,\text{ \ \ }\bar{\nabla}_{X}\xi =\bar{\nabla}%
_{X}N=0,\text{ \ \ }\forall X\in \Gamma \left( TM\right) .
\end{eqnarray*}%
where $\bar{\nabla}$ denotes\ the Levi-Civita connection on $R_{2}^{4}$.
Then using the Gauss and Weingarten formulae we find%
\begin{eqnarray*}
D_{1} &=&0;\text{ } \\
\text{\ \ }A_{\xi } &=&0; \\
\text{ \ \ }A_{N} &=&0;\text{ } \\
\text{\ \ }\nabla _{X}\xi &=&0; \\
\rho _{1}\left( X\right) &=&0; \\
D_{2}\left( X,\xi \right) &=&0; \\
D_{2}\left( U_{2},U_{2}\right) &=&2; \\
\nabla _{X}U_{2} &=&\frac{2\sqrt{2}\left( x^{2}-x^{1}\right) ^{3}}{1+\left(
x^{1}-x^{2}\right) ^{2}}X^{2}U_{2;}
\end{eqnarray*}%
for any $X=X^{1}\xi +X^{2}U_{2}$\ tangent to $M$ \cite{Krishan L. Duggal and
Aurel Bejancu}. Since $D_{1}=0$, it follows that the induced connection $%
\nabla $\ is a metric connection. Using
\begin{equation*}
\bar{g}\left( U_{2},U_{2}\right) =-\left( 1+\left( x^{1}-x^{2}\right)
^{4}\right)
\end{equation*}%
\ we have%
\begin{equation*}
D_{2}\left( U_{2},U_{2}\right) =H_{2}\bar{g}\left( U_{2},U_{2}\right) ,\text{
\ \ }H_{2}=-\frac{1}{\left( 1+\left( x^{1}-x^{2}\right) ^{4}\right) }.
\end{equation*}%
Therefore, $M$\ is a totally umbilical half-lightlike submanifold of $%
R_{2}^{4}.$Then by straightforward calculations we obtain%
\begin{equation*}
D_{2}\left( \xi ,\xi \right) =0.
\end{equation*}%
Therefore, the intersection of $M$\ and $E(p,\xi )$\ gives a lightlike curve
$\gamma $\ in a neighborhood of $p,$\ which is called the normal section of $%
M$\ at the point $p$\ in the direction of $\xi ,$\ namely%
\begin{eqnarray*}
\gamma ^{\prime }\left( s\right) &=&\xi , \\
\gamma ^{\prime \prime }\left( s\right) &=&\bar{\nabla}_{\xi }\xi =0.
\end{eqnarray*}%
Hence we obtain%
\begin{equation*}
\gamma ^{\prime \prime \prime }\left( s\right) \wedge \gamma ^{\prime \prime
}\left( s\right) \wedge \gamma ^{\prime }\left( s\right) =0.
\end{equation*}%
\
\end{example}

\subsection{Non- Degenerate Planar Normal Section in Half-Lightlike
Submanifolds}

\qquad

\qquad In this subsection we investigate the conditions for a screen
conformal half-lightlike hypersurface $M$ of $R_{2}^{4}$\ to have
non-degenerate planar normal sections.

\begin{theorem}
\bigskip\ $M$\ be a screen conformal half-lightlike hypersurface in $%
R_{2}^{4}$. $M$\ has spacelike planar normal sections if and only if%
\begin{equation}
T\left( v,v\right) \wedge \bar{\nabla }_{v}T\left( v,v\right) =0
\label{2.2.5}
\end{equation}%
where $v\in \Gamma (S(TM))$\ and $T\left( v,v\right) =E_{1}\left( v,v\right)
\xi +D_{1}\left( v,v\right) N+D_{2}\left( v,v\right) u.$
\end{theorem}

\begin{proof}
Let $M$ be a screen conformal half-lightlike submanifold and $\gamma$\ a
spacelike curve on $M$. Then we have
\begin{eqnarray}
\gamma ^{\prime }\left( s\right) &=&v\text{ ,}  \label{2.2.1} \\
\gamma ^{\prime \prime }\left( s\right) &=&\bar{\nabla }_{v}v=\nabla
_{v}^{\ast }v+E_{1}\left( v,v\right) \xi +D_{1}\left( v,v\right)
N+D_{2}\left( v,v\right) u\text{,}  \label{2.2.2} \\
\text{\ \ \ }\gamma ^{\prime \prime \prime }\left( s\right) &=&\nabla
_{v}^{\ast }\nabla _{v}^{\ast }v+E_{1}\left( v,\nabla _{v}^{\ast }v\right)
\xi +D_{1}\left( v,\nabla _{v}^{\ast }v\right) N+D_{2}\left( v,\nabla
_{v}^{\ast }v\right) u  \notag \\
&&+v\left( E_{1}\left( v,v\right) \right) \xi +v\left( D_{1}\left(
v,v\right) \right) N+v\left( D_{2}\left( v,v\right) \right) u  \notag \\
&&-E_{1}\left( v,v\right) A_{\xi }^{\ast }v+E_{1}\left( v,v\right)
u_{1}\left( v\right) \xi +E_{1}\left( v,v\right) D_{2}\left( v,\xi \right) u
\notag \\
&&-D_{1}\left( v,v\right) A_{N}v+D_{1}\left( v,v\right) \rho _{1}\left(
v\right) N+D_{1}\left( v,v\right) \rho _{2}\left( v\right) u  \notag \\
&&-D_{2}\left( v,v\right) A_{u}v+D_{2}\left( v,v\right) \varepsilon
_{1}\left( v\right) N\text{ \ \ \ }  \label{2.2.3}
\end{eqnarray}%
where $\nabla ^{\ast }$\ is the induced connection of $M^{\prime }$\ and $%
\gamma ^{\prime }\left( s\right) =v$, $\gamma ^{\prime }\left( 0\right)
=\upsilon .$\ From definition planar normal section and $S(TM)=Sp\left\{
v\right\} $ we have%
\begin{equation}
v\wedge \nabla _{v}^{\ast }v=0\text{ {\large and }}v\wedge \nabla _{v}^{\ast
}\nabla _{v}^{\ast }v=0.  \label{2.2.4}
\end{equation}%
Assume that $M$\ has planar non-degenerate normal sections. Then we have%
\begin{equation*}
\gamma ^{\prime \prime \prime }\left( s\right) \wedge \gamma ^{\prime \prime
}\left( s\right) \wedge \gamma ^{\prime }\left( s\right) =0.
\end{equation*}%
Thus from (\ref{2.2.4})
\begin{equation*}
E_{1}\left( v,v\right) \xi +D_{1}\left( v,v\right) N+D_{2}\left( v,v\right) u
\end{equation*}%
and
\begin{eqnarray*}
&&E_{1}\left( v,\nabla _{v}^{\ast }v\right) \xi +D_{1}\left( v,\nabla
_{v}^{\ast }v\right) N+D_{2}\left( v,\nabla _{v}^{\ast }v\right) u \\
&&+v\left( E_{1}\left( v,v\right) \right) \xi +v\left( D_{1}\left(
v,v\right) \right) N+v\left( D_{2}\left( v,v\right) \right) u \\
&&-E_{1}\left( v,v\right) A_{\xi }^{\ast }v+E_{1}\left( v,v\right)
u_{1}\left( v\right) \xi +E_{1}\left( v,v\right) D_{2}\left( v,\xi \right) u
\\
&&-D_{1}\left( v,v\right) A_{N}v+D_{1}\left( v,v\right) \rho _{1}\left(
v\right) N+D_{1}\left( v,v\right) \rho _{2}\left( v\right) u \\
&&-D_{2}\left( v,v\right) A_{u}v+D_{2}\left( v,v\right) \varepsilon
_{1}\left( v\right) N.
\end{eqnarray*}%
are linearly dependent. We put%
\begin{eqnarray*}
\bar{\nabla }_{v}T\left( v,v\right) &=&E_{1}\left( v,\nabla _{v}^{\ast
}v\right) \xi +D_{1}\left( v,\nabla _{v}^{\ast }v\right) N+D_{2}\left(
v,\nabla _{v}^{\ast }v\right) u \\
&&+v\left( E_{1}\left( v,v\right) \right) \xi +v\left( D_{1}\left(
v,v\right) \right) N+v\left( D_{2}\left( v,v\right) \right) u \\
&&-E_{1}\left( v,v\right) A_{\xi }^{\ast }v+E_{1}\left( v,v\right)
u_{1}\left( v\right) \xi +E_{1}\left( v,v\right) D_{2}\left( v,\xi \right) u
\\
&&-D_{1}\left( v,v\right) A_{N}v+D_{1}\left( v,v\right) \rho _{1}\left(
v\right) N+D_{1}\left( v,v\right) \rho _{2}\left( v\right) u \\
&&-D_{2}\left( v,v\right) A_{u}v+D_{2}\left( v,v\right) \varepsilon
_{1}\left( v\right) N
\end{eqnarray*}%
where $\gamma $\ is assumed to be parameterized by arc-length. Thus, we
obtain%
\begin{equation*}
T\left( v,v\right) \wedge \bar{\nabla }_{v}T\left( v,v\right) =0.
\end{equation*}%
Conversely, we assume that $T\left( v,v\right) \wedge \bar{\nabla }%
_{v}T\left( v,v\right) =0$\ for a spacelike tangent vector $v$\ of $M$\ at $%
p.$\ Then either $T\left( v,v\right) =0$\ or $\bar{\nabla }_{v}T\left(
v,v\right) =0.$\ If $T\left( v,v\right) =0,$\ then from (\ref{2.2.1}), (\ref%
{2.2.2}), (\ref{2.2.3}) and (\ref{2.2.4}), $M$\ has degenerate planar normal
sections. If $\bar{\nabla }_{v}T\left( v,v\right) =0,$\ from (\ref{2.2.4}),
we obtain%
\begin{equation*}
\gamma ^{\prime \prime \prime }\left( s\right) \wedge \gamma ^{\prime \prime
}\left( s\right) \wedge \gamma ^{\prime }\left( s\right) =v\wedge T\left(
v,v\right) \wedge \bar{\nabla }_{v}T\left( v,v\right) =0.
\end{equation*}
\end{proof}

\begin{example}
\ Let $M$ be a half-lightlike hypersurface of the 4- dimensional
semi-Riemann space $\left( R_{2}^{4},\overline{g}\right) $ of index 2, as
given in example\ref{ex1}. Now, for a point $p$\ in $M$\ and a spacelike
vector $U_{2}$\ tangent to $M$\ at $p$\ $\left( U_{2}\in S(TM\right) )$, the
vector $U_{2}$ \ and transversal space $tr(TM)$\ to $M$\ at $p$\ determine
an 2- dimensional subspace $E(p,U_{2})$\ in $R_{2}^{4}$\ through $p$. The
intersection of $M$\ and $E(p,U_{2})$\ gives a spacelike curve $\gamma $\ in
a neighborhood of $p$.\ Now we research half-lightlike hypersurfaces of $%
R_{2}^{4}$\ semi-Riemannian manifold have to condition non-degenerate planar
normal sections. Since Hence, we obtain%
\begin{eqnarray*}
\gamma ^{\prime } &=&U_{2}=\sqrt{2}\left( 1+\left( x^{1}-x^{2}\right)
^{2}\right) \partial _{1}+\left( 1+\left( x^{1}-x^{2}\right) ^{2}\right)
\partial _{3}-\sqrt{2}\left( x^{1}-x^{2}\right) \partial _{4}, \\
\gamma ^{\prime \prime } &=&\bar{\nabla}_{U_{2}}U_{2}=2\left( 1+\left(
x^{1}-x^{2}\right) ^{2}\right) .\left\{ 2\left( x^{2}-x^{1}\right) \partial
_{2}+\sqrt{2}\left( x^{2}-x^{1}\right) \partial _{3}+\partial _{4}\right\} ,
\\
\gamma ^{\prime \prime \prime } &=&\bar{\nabla}_{U_{2}}\bar{\nabla}%
_{U_{2}}U_{2}=\nabla _{U_{2}}\bar{\nabla}_{U_{2}}U_{2}+D\left( U_{2},\bar{%
\nabla}_{U_{2}}U_{2}\right) u \\
&=&\sqrt{2}\left( 1+\left( x^{1}-x^{2}\right) ^{2}\right) \left[
\begin{array}{c}
4\left( 1+3\left( \left( x^{1}-x^{2}\right) ^{2}\right) \right) \partial _{2}
\\
+2\sqrt{2}\left( 1+3\left( \left( x^{1}-x^{2}\right) ^{2}\right) \right)
\partial _{3}-4\left( x^{1}-x^{2}\right) \partial _{4}%
\end{array}%
\right]  \\
&&+\frac{4\sqrt{2}\left( x^{2}-x^{1}\right) ^{3}}{\left( 1+\left(
x^{1}-x^{2}\right) ^{4}\right) }\left( 1+\left( x^{1}-x^{2}\right)
^{2}\right) \left[
\begin{array}{c}
2\left( x^{2}-x^{1}\right) \partial _{2} \\
+\sqrt{2}\left( x^{2}-x^{1}\right) \partial _{3}+\left( 1+x^{1}-x^{2}\right)
\partial _{4}%
\end{array}%
\right] .
\end{eqnarray*}%
Then, by direct calculations we find%
\begin{eqnarray}
E_{1}\left( U_{2},U_{2}\right)  &=&0,  \label{1} \\
E_{1}\left( U_{2},\nabla _{U_{2}}^{\ast }U_{2}\right)  &=&0.  \label{2}
\end{eqnarray}%
Thus from $\left( \ref{1}\right) $ - $\left( \ref{2}\right) $ , $T\left(
U_{2},U_{2}\right) $ and $\bar{\nabla}_{U_{2}}T\left( U_{2},U_{2}\right) $
are linearly dependent.
\end{example}

\begin{example}
Consider a surface $M$\ in $R_{1}^{4}$\ given by the equation%
\begin{equation*}
x_{1}=x_{3},x_{2}=\left( 1-x_{4}\right) ^{\frac{1}{2}}
\end{equation*}%
Hence, the natural frames field on $M$\ is globally
\begin{eqnarray*}
X_{x_{1}} &=&\frac{\partial }{\partial x_{1}}+\frac{\partial }{\partial x_{3}%
} \\
X_{x_{4}} &=&-\frac{x_{4}}{x_{2}}\frac{\partial }{\partial x_{2}}+\frac{%
\partial }{\partial x_{4}}.
\end{eqnarray*}%
Then, we obtain
\begin{eqnarray*}
TM &=&Sp\left\{ \xi =\partial x_{1}+\partial x_{3},v=-x_{4}\partial
x_{2}+x_{2}\partial x_{4}\right\}  \\
TM^{\perp } &=&Sp\left\{ \xi =\partial x_{1}+\partial x_{3},u=x_{2}\partial
x_{2}+x_{4}\partial x_{2}\right\} .
\end{eqnarray*}%
Therefore, we have $RadTM=Sp\left\{ \xi \right\} $, $S(TM)=Sp\left\{
v\right\} $, $S(TM^{\perp })=Sp\left\{ u\right\} $\ and $ltr(TM)=Sp\left\{ N=%
\frac{1}{2}\left( \partial x_{1}+\partial x_{3}\right) \right\} $, which
show $M$\ is a half-lightlike submanifold of $R_{1}^{4}.$ Then using the
Gauss and Weingarten formulae we find
\begin{eqnarray*}
\bar{\nabla}_{v}\xi  &=&\nabla _{v}\xi +D_{2}\left( v,\xi \right) u=-A_{\xi
}^{\ast }v+D_{2}\left( v,\xi \right) u \\
\bar{g}\left( \nabla _{v}\xi ,v\right)  &=&-\bar{g}\left( A_{\xi }^{\ast
}v,v\right) ,
\end{eqnarray*}%
On account of $\nabla _{v}\xi =0$, we have
\begin{equation*}
\bar{g}\left( A_{\xi }^{\ast }v,v\right) =0\Rightarrow A_{\xi }^{\ast }v=0.
\end{equation*}%
Moreover, from (\ref{1.5}) and by straightforward calculations we obtain
\begin{eqnarray*}
\bar{\nabla}_{v}N &=&-A_{N}v+\rho _{1}\left( v\right) N+\rho _{2}\left(
v\right) u \\
\bar{g}\left( \bar{\nabla}_{v}N,v\right)  &=&-\bar{g}\left( A_{N}v,v\right)
\\
\bar{g}\left( N,\bar{\nabla}_{v}v\right)  &=&\bar{g}\left( A_{N}v,v\right)
\\
\bar{g}\left( N,\nabla _{v}v\right)  &=&\bar{g}\left( A_{N}v,v\right)
\end{eqnarray*}%
and%
\begin{eqnarray*}
\nabla _{v}v &=&\left( v\left[ 0\right] ,v\left[ -x_{4}\right] ,v\left[ 0%
\right] ,v\left[ x_{2}\right] \right) =\left( 0,-x_{2},0,-x_{4}\right)  \\
\bar{g}\left( N,\nabla _{v}v\right)  &=&\bar{g}\left( \frac{1}{2}\left(
-1,0,1,0\right) ,\left( 0,-x_{2},0,-x_{4}\right) \right) =\bar{g}\left(
A_{N}v,v\right) =0 \\
&\Rightarrow &A_{N}v=0
\end{eqnarray*}%
\ or $A_{N}v\in RadTM$. Using (\ref{1.4}), we have%
\begin{eqnarray*}
\bar{\nabla}_{v}v &=&\nabla _{v}v+D_{1}\left( v,v\right) N+D_{2}\left(
v,v\right) u \\
&\Rightarrow &D_{1}\left( v,v\right) =\bar{g}\left( \bar{\nabla}_{v}v,\xi
\right) =-\bar{g}\left( v,\bar{\nabla}_{v}\xi \right)  \\
&\Rightarrow &D_{1}\left( v,v\right) =\bar{g}\left( A_{\xi }^{\ast
}v,v\right) =0
\end{eqnarray*}%
and since $D_{1}\left( v,\xi \right) =0,$\ we have $D_{1}=0$. Using by (\ref%
{1.4}), we obtain
\begin{equation*}
D_{2}\left( v,v\right) \epsilon =\bar{g}\left( \bar{\nabla}_{v}v,u\right) =-%
\bar{g}\left( v,\bar{\nabla}_{v}u\right) =\bar{g}\left( v,A_{u}v\right) .
\end{equation*}%
Since
\begin{equation*}
\bar{g}(\bar{\nabla}_{v}\bar{\nabla}_{v}v,N)=\bar{g}\left( A_{u}v,N\right) =0
\end{equation*}%
\ $A_{u}v\in S(TM)$\ and by straightforward calculations we obtain
\begin{eqnarray*}
\bar{\nabla}_{v}\bar{\nabla}_{v}v &=&\nabla _{v}\nabla _{v}v+D_{2}\left(
v,\nabla _{v}v\right) u+A_{u}v-\varepsilon _{1}\left( u\right) N \\
\bar{g}\left( \bar{\nabla}_{v}\bar{\nabla}_{v}v,u\right)  &=&D_{2}\left(
v,\nabla _{v}v\right) \epsilon =-\bar{g}\left( v,\nabla _{v}v\right)
=2x_{2}x_{4} \\
\bar{g}\left( \bar{\nabla}_{v}\bar{\nabla}_{v}v,\xi \right)  &=&-\varepsilon
_{1}\left( u\right)  \\
\rho _{1}\left( v\right)  &=&\bar{g}\left( \bar{\nabla}_{v}N,\xi \right) =-%
\bar{g}\left( A_{N}v,\xi \right) =0 \\
\rho _{2}\left( v\right)  &=&\epsilon \bar{g}\left( \bar{\nabla}%
_{v}N,u\right) =-\epsilon \bar{g}\left( A_{N}v,u\right) =0 \\
\varepsilon _{1}\left( u\right)  &=&\bar{g}\left( \bar{\nabla}_{v}v,\bar{%
\nabla}_{v}\xi \right) =0\Rightarrow D_{2}\left( v,\xi \right) =0 \\
\epsilon D_{2}\left( v,v\right)  &=&\bar{g}\left( \bar{\nabla}_{v}v,u\right)
=-\bar{g}\left( v,A_{u}v\right) =-\bar{g}\left( v,v\right) =-1.
\end{eqnarray*}%
Let $v\in S(TM)$\ and $p\in M$.\ We denote subspace%
\begin{equation*}
E\left( p,v\right) =\left\{ v\right\} \cup tr\left( TM\right) .
\end{equation*}%
and we have
\begin{equation*}
E\left( p,v\right) \cap M=\gamma .
\end{equation*}%
where $\gamma $\ is the normal section of $M$\ at $p$\ in the direction of $v
$. Then we have
\begin{eqnarray*}
\gamma ^{\prime }\left( s\right)  &=&v=-x_{4}\partial x_{2}+x_{2}\partial
x_{4} \\
\gamma ^{\prime \prime }\left( s\right)  &=&\bar{\nabla}_{v}v=\nabla
_{v}v+D_{2}\left( v,v\right) u=-2x_{2}\partial x_{2}-2x_{4}\partial x_{4} \\
\gamma ^{\prime \prime \prime }\left( s\right)  &=&\bar{\nabla}_{v}\bar{%
\nabla}_{v}v=\left( 2x_{4}+4x_{2}^{2}x_{4})\partial
x_{2}+(-2x_{2}+2x_{2}x_{4}^{2}\right) \partial x_{4}.
\end{eqnarray*}%
Hence
\begin{equation*}
\gamma ^{\prime \prime \prime }\left( s\right) \wedge \gamma ^{\prime \prime
}\left( s\right) \wedge \gamma ^{\prime }\left( s\right) =0,
\end{equation*}%
that is $M$\ has non-degenerate planar normal sections.
\end{example}

\begin{proposition}
\label{po.3.1}\ Let $M$\ be a half-lightlike hypersurface in $R_{2}^{4}.$\
If $M$\ has planar normal sections, then%
\begin{equation}
\nabla _{v}^{\ast }v=0  \label{2.2.6}
\end{equation}%
where $\gamma $\ is normal section in the direction $v=\gamma ^{\prime
}\left( s\right) $\ for $v\in \Gamma \left( S(TM)\right) .$
\end{proposition}

\begin{proof}
From $v\in S(TM)$\ we have%
\begin{equation}
\left\langle v,v\right\rangle =1\Rightarrow \left\langle v,\nabla _{v}^{\ast
}v\right\rangle =0.  \label{2.2.7}
\end{equation}%
Using \ the definition of normal section and (\ref{2.2.7}) we complete the
proof.
\end{proof}

\bigskip

Now we define a function $L$\ by%
\begin{equation*}
L(p,v)=L_{p}\left( v\right) =\left\langle T(v,v),T(v,v)\right\rangle
\end{equation*}%
on $\cup _{p}M,$\ where $\bigcup _{p}M=\left\{ v\in \Gamma \left( TM\right)
\mid \left\langle v,v\right\rangle ^{\frac{1}{2}}=1\right\} .$\ If $L\neq 0,$%
\ then $M$\ has non-degenerate pointwise normal sections. By a vertex of
curve $\gamma $\ we mean a point $p$\ on $\gamma $\ such that its curvature $%
\kappa $\ satisfies $\frac{d\kappa ^{2}\left( 0\right) }{ds}=0.$\ Let $M$\
has planar normal sections. From proposition\ref{po.3.1} we obtain%
\begin{eqnarray*}
\epsilon \kappa ^{2}\left( s\right) &=&2E_{1}\left( v,v\right) D_{1}\left(
v,v\right) +D_{2}^{2}\left( v,v\right) \epsilon , \\
\frac{1}{2}\frac{d\kappa ^{2}\left( 0\right) }{ds} &=&v(E_{1}\left(
v,v\right) D_{1}\left( v,v\right) )+v(D_{2}\left( v,v\right) )D_{2}\left(
v,v\right) \epsilon .
\end{eqnarray*}%
If $M$\ is totally geodesic, then $D_{1}=D_{2}=0.$\ Thus $\gamma $\ has a
vertex.

\bigskip

Consequently, we have the following

\begin{theorem}
\label{th.3.9}Let $M$\ be a half-lightlike hypersurface of $R_{2}^{4}.$\ If $%
M$\ has non-degenerate planar normal sections and submanifold is totally
geodesic at $p\in M,$\ then normal section curve $\gamma $\ has a vertex at $%
p\in M$.
\end{theorem}

\begin{theorem}
\ Let $M$\ be a half-lightlike hypersurface of $R_{2}^{4}.$\ with\ planar
normal sections. Then normal section curve $\gamma $\ has a vertex and
submanifold is totally geodesic if and only if M is minimal.
\end{theorem}

\begin{proof}
\ If $M$\ is totally geodesic, then from ($tr\mid _{S(TM)}h=0)$\ and $%
\varepsilon _{1}\left( \xi \right) =0$, we conclude.
\end{proof}

\bigskip

From theorem\ref{th.3.7} and theorem\ref{th.3.9}, we give

\begin{theorem}
\ Let $M$\ be a half-lightlike hypersurface in $R_{2}^{4}\left( c\right) $\
with planar normal sections. Then $K_{\xi }\left( H\right) =0$\ if and only
if normal section curve $\gamma $\ has a vertex at $p\in M$ where $\xi \in
\Gamma \left( RadTM\right) $
\end{theorem}

\begin{theorem}
Let $M$\ be a half-lightlike hypersurface of $R_{2}^{4}$\ and the normal
section $\gamma $\ at for any $p$\ be a geodesic arc on a sufficiently small
neighborhood of $p$. Then $M$\ has non-degenerate planar normal sections if
and only if\
\begin{equation*}
h\left( v,v\right) \wedge (\bar{\nabla}_{v}h)\left( v,v\right) =0
\end{equation*}%
where is $h\left( v,v\right) =D_{1}\left( v,v\right) N+D_{2}\left(
v,v\right) u.$
\end{theorem}

\begin{proof}
If normal section $\gamma $\ at for any $p$\ is a geodesic arc on a
sufficiently small neighborhood of $p$, we have
\begin{eqnarray*}
\gamma ^{\prime }\left( s\right) &=&v \\
\gamma ^{\prime \prime }\left( s\right) &=&D_{1}\left( v,v\right)
N+D_{2}\left( v,v\right) u \\
\gamma ^{\prime \prime \prime }\left( s\right) &=&v(D_{1}\left( v,v\right)
)N+v(D_{2}\left( v,v\right) )u \\
&&-D_{1}\left( v,v\right) A_{N}v+D_{1}\left( v,v\right) \rho _{1}\left(
v\right) N \\
&&+D_{1}\left( v,v\right) \rho _{2}\left( v\right) u-D_{2}\left( v,v\right)
A_{u}v \\
&&+D_{2}\left( v,v\right) \varepsilon _{1}\left( v\right) N.
\end{eqnarray*}%
Since $\gamma $\ is a planar curve then we get
\begin{equation*}
v\wedge \left( D_{1}\left( v,v\right) N+D_{2}\left( v,v\right) u\right)
\wedge \left(
\begin{array}{c}
v(D_{1}\left( v,v\right) )N+v(D_{2}\left( v,v\right) )u \\
-D_{1}\left( v,v\right) A_{N}v+D_{1}\left( v,v\right) \rho _{1}\left(
v\right) N \\
+D_{1}\left( v,v\right) \rho _{2}\left( v\right) u-D_{2}\left( v,v\right)
A_{u}v \\
+D_{2}\left( v,v\right) \varepsilon _{1}\left( v\right) N%
\end{array}%
\right) =0.
\end{equation*}%
Therefore, by taking the covariant derivative of
\begin{equation*}
h\left( v,v\right) =D_{1}\left( v,v\right) N+D_{2}\left( v,v\right) u,
\end{equation*}%
\ we obtain
\begin{equation*}
(\bar{\nabla}_{v}h)\left( v,v\right) =\bar{\nabla}_{v}h\left( v,v\right)
=\gamma ^{\prime \prime \prime }\left( s\right) .
\end{equation*}%
which gives
\begin{equation*}
\gamma ^{\prime \prime \prime }\left( s\right) \wedge \gamma ^{\prime \prime
}\left( s\right) \wedge \gamma ^{\prime }\left( s\right) =v\wedge h\left(
v,v\right) \wedge (\bar{\nabla}_{v}h)\left( v,v\right) =0.
\end{equation*}%
From the last equation above, we have
\begin{equation*}
h\left( v,v\right) \wedge (\bar{\nabla}_{v}h)\left( v,v\right) =0.
\end{equation*}%
Conversely, we assume that $h\left( v,v\right) \wedge (\bar{\nabla}%
_{v}h)\left( v,v\right) =0$. In this case, we have either $h\left(
v,v\right) =0$\ or $(\bar{\nabla}_{v}h)\left( v,v\right) =0$. If $h\left(
v,v\right) =0$, we have $D_{1}\left( v,v\right) =0$\ and $D_{2}\left(
v,v\right) =0$. In this way, we get
\begin{eqnarray*}
\bar{\nabla}_{\xi }v &=&-A_{v}\xi +\varepsilon _{1}\left( \xi \right) N \\
&\Rightarrow &\bar{g}\left( \bar{\nabla}_{\xi }v,\xi \right) =-\bar{g}\left(
A_{v}\xi ,\xi \right) +\varepsilon _{1}\left( \xi \right) \\
0 &=&\varepsilon _{1}\left( \xi \right)
\end{eqnarray*}%
which shows that $M$\ is minimal and \ has planar normal sections. On the
other hand if $(\bar{\nabla}_{v}h)\left( v,v\right) =0,$\ from $(\bar{\nabla}%
_{v}h)\left( v,v\right) =\bar{\nabla}_{v}h\left( v,v\right) =\gamma ^{\prime
\prime \prime }\left( s\right) =0,$\ we obtain
\begin{equation*}
\gamma ^{\prime \prime \prime }\left( s\right) \wedge \gamma ^{\prime \prime
}\left( s\right) \wedge \gamma ^{\prime }\left( s\right) =0,
\end{equation*}%
that is $M$\ has non-degenerate planar normal sections.
\end{proof}

We also have the following result

\begin{theorem}
Let $M$\ be a half-lightlike hypersurface in $R_{2}^{4}$\ and the normal
section $\gamma $\ at for any $p$\ be a geodesic arc on a sufficiently small
neighborhood of $p$. Then the following statements are equivalent;

\begin{enumerate}
\item $(\bar{\nabla }_{v}h)\left( v,v\right) =0,$

\item $\bar{\nabla }h=0,$

\item $M$\ has non degenerate planar normal sections of $p\in M$\ and $%
\gamma $\ has vertex point at $p$\ $\in M,$

\item \ $D_{2}=0$ in $S(TM).$
\end{enumerate}
\end{theorem}

\begin{proof}
For curvature $\kappa $\ at $p$\ point of $\gamma $, we have%
\begin{eqnarray}
\epsilon \kappa ^{2}\left( s\right) &=&\left\langle \gamma ^{\prime \prime
}\left( s\right) ,\gamma ^{\prime \prime }\left( s\right) \right\rangle
\notag \\
&=&D_{2}^{2}\left( v,v\right) \epsilon  \notag \\
\frac{1}{2}\epsilon \frac{d\kappa ^{2}\left( s\right) }{ds} &=&v\left(
D_{2}\left( v,v\right) \right) D_{2}\left( v,v\right) \epsilon  \label{2.2.8}
\end{eqnarray}%
and $\ $from $\epsilon \kappa ^{2}\left( s\right) =\left\langle \gamma
^{\prime \prime }\left( s\right) ,\gamma ^{\prime \prime }\left( s\right)
\right\rangle $\
\begin{eqnarray}
\frac{1}{2}\epsilon \frac{d\kappa ^{2}\left( s\right) }{ds} &=&\left\langle
\gamma ^{\prime \prime \prime }\left( s\right) ,\gamma ^{\prime \prime
}\left( s\right) \right\rangle  \notag \\
&=&\left\langle (\bar{\nabla }_{v}h)\left( v,v\right) ,h\left( v,v\right)
\right\rangle  \notag \\
&=&0.  \label{2.2.9}
\end{eqnarray}%
Hence, from (\ref{2.2.8}) and (\ref{2.2.9}), we obtain $D_{2}\left(
v,v\right) =0$. From here, we complete the proof.
\end{proof}

\bigskip

\end{document}